\documentclass[12pt,reqno]{amsart}
\usepackage{amscd,amssymb,amsmath,color}
  \usepackage[all]{xy}
\usepackage{url}
\numberwithin{equation}{section}

\input{prepictex.tex}
\input{pictex.tex}
\input{postpictex.tex}

\def\C{\mathbb C}
\def\N{\mathbb N}

\def\T{\mathbb T}
\def\Z{\mathbb Z}

\def\id{\operatorname{id}}

\def\pol{\mathcal{O}}
\def\lra{\longrightarrow}

\def\lra{\longrightarrow}

\def\sw#1{{\sb{(#1)}}}

\def\ot{\mathop{\otimes}}

\def\eps{\varepsilon}

\def\can{\mathsf{can}}

\def\k{\mathbb{K}}
\def\coC{{\mathrm{co}\,C}}
\def\coD{{\mathrm{co}\,D}}
\def\can{\mathrm{can}}
\def\lto{\longmapsto}

\newtheorem{theo}{Theorem}[section]

\newtheorem{lemm}[theo]{Lemma}
\newtheorem{prop}[theo]{Proposition}
\theoremstyle{definition}
\newtheorem{defi}[theo]{Definition}
\newtheorem{rema}[theo]{Remark}

\newcounter{zlist}
\newenvironment{zlist}{\begin{list}{(\arabic{zlist})}{
\usecounter{zlist}\leftmargin2.5em\labelwidth2em\labelsep0.5em
\topsep0.6ex
\parsep0.3ex plus0.2ex minus0.1ex}}{\end{list}}
 
\newcounter{blist}
\newenvironment{blist}{\begin{list}{(\alph{blist})}{
\usecounter{blist}\leftmargin2.5em\labelwidth2em\labelsep0.5em
\topsep0.6ex
\parsep0.3ex plus0.2ex minus0.1ex}}{\end{list}}
 
\newcounter{rlist}

\title[Noncommutative homogeneous bundles]{An algebraic framework for noncommutative bundles with homogeneous fibres}

 \author{Tomasz Brzezi\'nski}
 \address{ Department of Mathematics, Swansea University, 
Swansea University Bay Campus,
Fabian Way,
Swansea,
  Swansea SA1 8EN, U.K.\ 
\& Department of Mathematics, University of Bia\l ystok, K.\ Cio\l kowskiego 1M, 15--245 Bia\l ystok, Poland} 
  \email{T.Brzezinski@swansea.ac.uk}   
\author{Wojciech Szyma\'nski}
\address{Department of Mathematics and Computer Science, University of Southern Denmark, 
Campusvej 55, 5230 Odense M, Denmark} 
\email{szymanski@imada.sdu.dk} 

\thanks{The second named author was  supported by  the 
DFF-Research Project 2, `Automorphisms and invariants of operator algebras', Nr. 7014--00145B, 2017--2021. }

\date\today

\begin{document}

\begin{abstract}
An algebraic framework for noncommutative bundles with (quantum) homogeneous fibres is proposed.  The framework relies on the use of principal coalgebra extensions which play the role of principal bundles  in noncommutative geometry which might be additionally equipped with a Hopf algebra symmetry. The proposed framework is supported by two examples of noncommutative $\mathbb{C} P_q^1$-bundles: the quantum flag manifold viewed as a bundle with a generic Podle\'s sphere as a fibre, and the quantum twistor bundle viewed as a bundle over the quantum 4-sphere of Bonechi, Ciccoli and Tarlini.
\end{abstract}

\maketitle

\section{Introduction}
Principal bundles in noncommutative geometry or fibre bundles with quantum group  fibres have been well understood since at least the 1990s. Starting with the pioneering work of Schneider \cite{Sch:pri}, through more geometric approach in \cite{brzma}, \cite{h} to example-driven generalisations \cite{BrzMaj:coa}, \cite{BrzMaj:geo}, the algebraic notion of a noncommutative principal bundle has been formalised as a {\em principal coalgebra extension}; see e.g.\ \cite{brzha0}. The aim of this paper is to drop some of the symmetry of a fibre and develop an algebraic framework for noncommutative bundles with quantum homogeneous fibres.  

In a recent work, the authors presented an interpretation of the quantum flag manifold as the total space of a noncommutative bundle over the quantum projective plane $\C P_q^2$ with the homogeneous fibre $\C P_q^1$, see \cite{BrzSzy:fla} and \cite{BrzSz:BiaProc}.  In here, rather than focussing on a particular example, we develop an algebraic framework, which not only captures the quantum flag manifold example, but it is also applicable to much more general situations. The key idea is that in order to compensate for the lack of the (quantum) group structure of the fibre, dually encoded in the comultiplication, one needs to keep a principal bundle in the background, whose fibres act transitively on the homogeneous fibres of the constructed bundles. We illustrate this framework by two examples. The first example is the quantum flag manifold  now fibered by generic Podle\'s spheres \cite{Pod:sph} or two-parameter projective lines  $\C P_{q,s}^1$, thus generalising the algebraic part of  \cite{BrzSzy:fla}, where the standard Podle\'s sphere was considered, but not yet fully exploring the framework developed in the present article. The second example, which now explores the full generality of our set-up, is  the quantum twistor bundle, i.e.\ the $\C P_q^1$-fibration of the quantum projective space $\C P_q^3$ with the quantum 4-sphere of Bonechi-Ciccoli-Tarlini \cite{BCT} as the base. The topological aspects of this example are studied elsewhere, see \cite{MSz}.

In the algebraic language suggested by the Gelfand duality, a fibration of a space corresponds to an inclusion of algebras. In the noncommutative world, 
however, one is immediately faced with formidable difficulties when trying to interpret algebraically classical notions of fibre and of local triviality, since both  
require  references to points in the space. Several approaches to these problems have been proposed in the literature, 
just to mention the concept of local triviality from \cite{hmsz} or piecewise triviality from \cite{hkmz}. Nevertheless, we are still some way from 
resolving these issues in a completely satisfactory fashion. Furthermore, an attempt to understand noncommutative fibre bundles in the spirit of Steenrod, 
\cite{st}, would require building structural symmetries from the fibration. This however does not seem possible for noncommutative spaces in view 
of lack of transition functions. In the present paper, we propose a way out of this predicament by 
placing emphasis on the cotensor product decomposition of the ambient algebras,  
see property (1) in Theorem \ref{thm.main} and property (2) in Theorem \ref{thm.main.DK}. In this way we can interpret the  
algebra of the (noncommutative) total space of the fibration as the algebra of sections of a bundle associated to a principal bundle. This combined 
with projectivity of this algebra over the  algebra of the base space, provides an adequate replacement for the local triviality condition. At the 
same time, it gives an algebraic way of recovering the fibre. 

The paper is organised as follows. In Section~\ref{sec.results} we present the set-up and state main results of the paper. Section~\ref{sec.ex} describes two examples of noncommutative bundles with homogeneous fibres mentioned above. In the final Section~\ref{sec.proofs} we give all the technical details of the proofs of main Theorems~\ref{thm.main} and \ref{thm.main.DK}.

\section{The framework: results}\label{sec.results}
In this section we present the algebraic set-up for noncommutative bundles and state the main results. We work over a field $\k$ and by an algebra (resp.\ coalgebra) we mean a unital associative algebra over $\k$ (resp.\ countial coassociative coalgebra over $\k$). Unadorned tensor product symbol $\otimes$ denotes the tensor product of $\k$-vector spaces. The identity in any algebra is denoted by 1. For any coalgebra $C$ the comultiplication is denoted by $\Delta$ and counit by $\eps$ (or by $\Delta_C$, $\eps_C$ if we want to stress the particular coalgebra, e.g.\ when more than one coalgebra appear in the same formula). The action of an algebra on a module different from an algebra multiplication is denoted by a dot in-between the elements. If $A$ and $B$ are algebras and $M$ and $N$ are right $A$- and $B$-modules respectively, then
$$
(m\ot n)\cdot (a\ot b) := m\cdot a \ot n\cdot b, \qquad \mbox{for all $a\in A$, $b\in B$, $m\in M$, $n\in N$};
$$
this is simply the formula for the action of the tensor product of algebras on the tensor product of their modules.

\begin{defi}\label{def.coal.Gal}
Let $C$ be a coalgebra and $P$ be an algebra and a right $C$-comodule with coaction $\varrho: P\to P\otimes C$. The subalgebra of {\em coinvariants} $P^\coC$ is defined as
$$
P^\coC:= \{b\in P\;|\; \forall p\in P\; \varrho(bp) = b\varrho(p)\},
$$
where the left $P$-action on $P\ot C$ is defined by $q(p\ot c) = qp\ot c$.  Set $B= P^\coC$. The extension of algebras $B\subseteq P$ is said to be {\em coalgebra-Galois} or  {\em $C$-Galois} if the map
$$
\can : P\ot_B P\lra P\ot C, \qquad p\ot q \lto p\varrho(q),
$$
known as the {\em canonical Galois map}, is bijective. 

A $C$-Galois extension $B\subseteq P$ is said to be {\em copointed} (or $e$-{\em copointed}) if 
$$
\varrho(1) = 1\otimes e,
$$
for a (necessarily) group-like element $e\in C$. \hfill$\Box$
\end{defi}

The notion described in Definition~\ref{def.coal.Gal} has been introduced first in \cite{BrzMaj:coa} and then, in the stated generality, in \cite{BrzHaj:coa}. A few comments are now in order. First, note that  the coinvariants $B$ form a subalgebra  and that, by definition, the coaction $\varrho$ is a left $B$-module homomorphism, so that the  canonical  Galois map is well-defined. Second, the map $\can$ is both left $P$-linear and right $C$-colinear provided that its domain and codomain are equipped with the natural (obvious) left $P$-module structures given by the multiplication in $P$, and with the right $C$-comodule structures $\id\ot_B\varrho$ and $\id\ot \Delta$. Third, in the copointed case one can  show  
 that the coinvariant subalgebra coincides with the space of {\em $e$-coinvariants}, i.e.\ $B = P^\coC_e$, where
\begin{equation}\label{e-coinv}
P^\coC_e:=\{ b\in P\;|\; \varrho(b) = b\ot e\}.
\end{equation}
Four, as observed in \cite{BrzHaj:coa}, a coalgebra-Galois extension has an additional symmetry arising from the {\em canonical entwining map}:
\begin{equation}\label{can.entw}
\psi: C\otimes P\lto P\otimes C, \qquad c\otimes p\mapsto \can\left(\can^{-1}\left(1\otimes c\right)p\right).
\end{equation}
The properties of $\psi$ are not essential for the formulation of the main statements of this paper but only for their proofs, hence we postpone exploring them to Section~\ref{sec.proofs}. Here we only indicate that $\psi$ could be understood as a device which records the transfer of the right $P$-module structure on $P\ot_BP$ to $P\ot C$ in such a way that the canonical Galois map is a homomorphism of $P$-bimodules. Explicitly, since the map $\can$ is an isomorphism of left $P$-modules we can define the right $P$-action on $P\ot C$ by
\begin{equation}\label{right.pc}
(p\ot c)\cdot q := \can\left(\can^{-1}\left(p\ot c\right)q\right) = p\psi(c\ot q), \quad \mbox{for all $c\in C$, $p,q\in P$},
\end{equation}
where the last equality follows by the left $P$-linearity of the canonical map (and its inverse) and \eqref{can.entw}.
We note  further that in the $e$-copointed case,
\begin{equation}\label{coact.entw}
\varrho(p) = \psi(e\ot p), \qquad \mbox{for all $p\in P$},
\end{equation}
since 
$$
\psi(e\ot p) = \can\left(\can^{-1}\left(1\otimes e\right)p\right) = \can(1\ot p) =\varrho(p).
$$
If $\psi$ is bijective, then $P$ is a left $C$-comodule, with the coaction $\lambda: P\to C\ot P$, which in the $e$-copointed case reads:
\begin{equation}\label{left.coact}
\lambda(p) = \psi^{-1}(p\ot e), \qquad \mbox{for all $p\in P$}.
\end{equation}
In that case $e$-coinvariants of the right $C$-coaction $\varrho$ coincide with $e$-coinvariants of the left $C$-coaction $\lambda$ \eqref{left.coact}, the latter defined symmetrically to \eqref{e-coinv}, 
\begin{equation}\label{e-coinv-left}
{}^\coC P_e:=\{ b\in P\;|\; \lambda(b) = e\ot b\}.
\end{equation}

The following definition, first stated in \cite{BrzHaj:Che}, captures most of the characteristics which could be expected of the 
noncommutative generalisations of a principal bundle.
\begin{defi}\label{def.princ}
An $e$-pointed $C$-Galois extension $B\subseteq P$ is called {\em principal coalgebra extension} provided the canonical entwining map \eqref{can.entw} is bijective and $P$ is a $C$-equivariantly projective left $B$-module, i.e.\ there exists a left $B$-module and right $C$-comodule splitting of the (restriction of the) multiplication map $B\ot P\to P$. \hfill$\Box$
\end{defi}

One should note that the notion of a principal coalgebra extension is left-right symmetric; although one initially starts with a right coaction and derives the left one, one can equally well start with the left coaction and derive the right one. Right coalgebra extension is principal if and only if the derived left coalgebra extension is principal and vice versa. Important consequences of Definition~\ref{def.princ} include that $P$ is projective both as a left and right $B$-module (but typically not as a $B$-bimodule), and it is also faithfully flat as a left and right $B$-module.  From the geometric point of view $P$ can be understood as the algebra of functions on the total space of a principal fibre bundle with the base, whose algebra of functions is $B$ and with the structure group whose product and identity  are encoded in the coproduct and counit of $C$.

In this paper we are interested in developing a framework for dealing with bundles with homogeneous spaces as fibres. Thus in addition to the data in Definition~\ref{def.princ} we consider a coalgebra $D$ and a coalgebra morphism 
\begin{equation}\label{pi} 
\pi: C\to D.
\end{equation}
Since $e\in C$ is a group-like element, $\bar{e}:=\pi(e)$ is a group-like element of $D$. The coaction $\varrho$ of $C$ on $P$ can be pushed to the coaction $\bar\varrho$ of $D$ on $P$,
\begin{equation}\label{coact.D}
\bar\varrho: P\lto P\ot D, \qquad \bar\varrho = (\id\ot \pi)\circ \varrho,
\end{equation}
and then one may consider the $\bar{e}$-coinvariants,
\begin{equation}\label{coinv.be}
A := P^\coD_{\bar{e}} = \{a\in P\; |\; \bar\varrho(a) = a\ot \bar{e}\}.
\end{equation}
In the set-up of Definition~\ref{def.princ}, $A$ is a left $B$-submodule  of $P$ that contains $B$. The $B$-module $A$ plays the role of the total space of the  bundle with homogeneous fibre. The latter is represented by the space:
\begin{equation}\label{x}
X:= C^\coD_{\bar{e}} = \{x\in C\; |\; (\id\ot \pi)\circ \Delta(x) = x\ot \bar{e}\}.
\end{equation}
Admittedly, $X$ is not an algebra (hence, in this generality, it can hardly be interpreted as functions on some geometric object), but it is a left $C$ coideal, i.e. $\Delta(X)\subseteq C\ot X$, which reflects the homogeneity of the underlying object. This, in particular, allows us to consider the cotensor products
\begin{equation}\label{coten.pcx}
P\,\Box_C X := \left\{\sum_i p_i\ot x_i\in P\ot X\; |\; \sum_i \varrho(p_i)\ot x_i = \sum_i p_i\ot \Delta(x_i)\right\}, \text{and} 
\end{equation}
\begin{equation}\label{coten.pdx}
P\,\Box_D X:= \left\{\sum_i p_i\ot x_i\in P\ot X \; |\; \sum_i \bar\varrho(p_i)\ot x_i = \sum_i p_i\ot \left((\pi\ot \id)\circ \Delta(x_i)\right)\right\}.
\end{equation}
Clearly, $P\,\Box_C X\subseteq P\,\Box_D X$. Since all elements of $B$ are coinvariant, the left $B$-action on $P\ot X$  restricts to the actions  on $P\,\Box_C X$ and $P\,\Box_D X$.  

With all this information at hand we can now state the first main result of this paper.

\begin{theo}\label{thm.main}
Let $B\subseteq P$ be a principal coalgebra $C$-extension. Let $\pi: C\to D$ be a coalgebra morphism and $A$ and $X$ the coinvariants of the induced coactions as defined in \eqref{coinv.be} and \eqref{x}. Then
\begin{zlist}
\item The coaction $\varrho$ restricts to the isomorphism of left $B$-modules
\begin{equation}\label{iso.apx}
A\cong P\,\Box_C X.
\end{equation}
\item $A$ is a projective left  $B$-module.
\item The canonical map $\can: P\ot_BP\to P\ot C$ restricts to the isomorphism of left $P$-modules
\begin{equation}\label{iso.papx}
P\ot_B A\cong P\ot X.
\end{equation}
\item The canonical map $\can: P\ot_BP\to P\ot C$ restricts to the isomorphism
\begin{equation}\label{iso.aapx}
\bar{A}\ot_B A\cong {}^\coD(P\ot X)_{\bar e},
\end{equation}
where 
$$
\bar{A} = \{a\in P\; |\; (\pi\ot \id)\circ \lambda (a) = \bar{e}\ot a\},
$$
and the $\bar e$-coinvariants on the right hand side of \eqref{iso.aapx} are calculated with respect to the left coaction
\begin{equation}\label{mixed.coact}
\Lambda: P\ot X \lto D\ot P\ot X, \qquad \Lambda = (\pi\ot\id\ot\id)\circ (\psi^{-1}\ot \id)\circ (\id\ot \Delta).
\end{equation}
\end{zlist}
\end{theo}

Let us make a few comments about the meaning of statements of Theorem~\ref{thm.main}. The first statement means that $A$ is a module of sections of a fibre bundle associated to the principal bundle represented by $P$; it indicates the position of $X$ in relation to $A$, and heuristically, allows one to at least informally say that $A$ is fibered by $X$. The second statement means that we can indeed interpret $A$ as (sections of) a bundle over (the space represented by) $B$. Furthermore, it also affirms existence of a connection in the sense of Cuntz and Quillen \cite{CunQui:alg}. The last two statements are stepping stones on the path leading to formulation of necessary conditions for $A$ to be interpreted as a bundle in terms of $A$, $B$ and the fibre $X$ alone. Removing the total algebra of an ambient principal bundle from one side, allows one to replace $C$ by the fibre $X$, and hence the full $C$ plays no role  in \eqref{iso.papx}. The facts that in general $\bar{A}$ is not necessarily isomorphic with $A$ (both contain $B$ as a submodule though) and that both appear on the left hand side of \eqref{iso.aapx} seem to suggest that in a noncommutative situation one perhaps should represent the total space of a fibre bundle by a pair of closely related modules. With this possibility in mind the left hand side of \eqref{iso.aapx} contains only the objects representing total and base spaces of the bundle with fibre $X$. Ideally, we would like all these objects be algebras (not merely modules over the base $B$).  The interpretation of the right hand side of \eqref{iso.aapx} is not that clear (except the fact that the fibre appears there).

It is possible for $A$ to be an algebra if some additional conditions of somewhat technical nature are imposed; these  are discussed in Remark~\ref{rem.Gal}. On the other hand this is also the case and even more can be said when there is a background symmetry encoded by a Hopf algebra. 
\begin{theo}\label{thm.main.DK}
Let $B\subseteq P$ be a principal coalgebra $C$-extension. Let $\pi: C\to D$ be a coalgebra morphism and $A$ and $X$ the coinvariants of the induced coactions as defined in \eqref{coinv.be} and \eqref{x}. Assume further that $H$ is a Hopf algebra with a bijective antipode such that
\begin{blist}
\item  $P$ is a right $H$-comodule algebra with coaction $\delta: P\to P\ot H$;
\item $C$ is a right $H$-module coalgebra, that is the right $H$-action on $C$  satisfies the conditions, for all $h\in H$ and $c\in C$,
\begin{equation}\label{mod.coal}
\Delta_C(c\cdot h) = \Delta_C(c)\cdot \Delta_H(h) \quad \mbox{and} \quad \eps_C(c\cdot h) = \eps_C(c)\eps_H(h);
\end{equation}
\item $D$ is a right $H$-module and $\pi: C\to D$ is a right $H$-module homomorphism;
\item the canonical Galois map is a right $P$-module homomorphism, when $P\ot C$ is equipped with the diagonal right $P$-action,
$$
(p\ot c)\cdot q = (p\ot c)\cdot\delta(q), \qquad \mbox{for all $p,q\in P$ and $c\in C$}.
$$
\end{blist}
Then
\begin{zlist}
\item $A$ is a subalgebra of $P$ containing $B$.
 \item The canonical Galois map restricts to the isomorphism:
\begin{equation}\label{iso.aapdx}
A\ot_B A\cong P\,\Box_D X .
\end{equation}
\end{zlist}
\end{theo}

\begin{rema}\label{rem.DK}
In view of \eqref{right.pc}, the condition (d) means that the canonical entwining map is of the {\em Doi-Koppinen type} (see \cite{Doi:uni}, \cite{Kop:var}, \cite{Brz:mod}) i.e.\
\begin{equation}\label{DK.entw}
\psi(c\ot p) = \sum_i p_i \ot c\cdot h_i, \qquad \mbox{where $c\in C$, $p\in P$, and $\sum_i p_i \ot h_i =\delta(p)$.}
\end{equation}
Indeed, 
$$
\psi(c\ot p) = \can(\can^{-1}(1\ot c)p) = \can\left(\can^{-1}\left(\left(1\ot c\right)\cdot p\right)\right) = (1\ot c)\cdot \delta(p),
$$
thus yielding \eqref{DK.entw}. The 
 inverse of $\psi$ comes out as 
$$
\psi^{-1}(p \ot c) =  \sum_i  c\cdot S^{-1}(h_i)\ot p_i,
$$
where $S$ is the antipode of $H$.
\end{rema}

The situation of $H=C$, $\varrho=\delta$ and $e=1_H$, so that $\bar e = \pi(1_H)$, is a special case. Since $P$ is a right $H$-comodule algebra the canonical entwining map and its inverse come out as
$$
\begin{aligned}
\psi: H\ot P\lra P\ot H, \qquad & h\ot p \lto \sum_ip_i \ot h h_i,\\
\psi^{-1} P\ot H \lra H\ot P, \qquad & p\ot h \lto  \sum_i h S^{-1}(h_i)\ot p_i,
\end{aligned}
$$
where, $\varrho(p) = \delta(p) = \sum_i p_i \ot h_i$. Since $\pi: H\to D$ is a homomorphism of right $H$-modules and the comultiplication in a Hopf algebra is an algebra homomorphism, the coinvariants $X = H^\coD_{\bar e}$ form a subalgebra of $H$. As before $\Delta(X)\subseteq H\ot X$. Thus $X$ is a left coideal subalgebra of $H$ or, in more geometric terms, an algebra of functions on a noncommutative or quantum homogeneous space.

\begin{rema}\label{rem.Gal}
In the situation of Theorem~\ref{thm.main}, for $A$ to be an algebra, it is sufficient that the canonical Galois map for the extension $P^{\coD} \subseteq P$ be injective. Likewise, 
 if $C$ has algebra structure and the canonical Galois map for the extension $C^{\coD} \subseteq C$ is injective, then $X$ is an algebra.
One might expect that the injectivity of the Galois maps will follow provided the map $\pi$ be surjective as  then  corresponding to the classical fact that restriction of a free action to a subgroup remains free. In the non-commutative setup presented here, however, the situation is far from being so straightforward, as the detailed discussion of Schauenburg and Schneider \cite{SchSch:gen} indicates. Corollary~4.6  of \cite{SchSch:gen} might be of particular interest. It states sufficient conditions for the canonical Galois map for the extension $P^{\coD} \subseteq P$ to be injective (to be a principal coalgebra extension in fact), and thus in particular for $A=P^{\coD}$ to be an algebra. These are:
\begin{blist}
\item  $\pi:C\to D$ is a surjective coalgebra map;
\item  $\pi$ splits as a right $D$-comodule map and $C$ is injective as a right $D$-comodule;
\item  $\psi (\ker \pi\ot P)\subseteq \ker ((\id\ot \pi)\circ \psi)$, and the induced map
$$
\theta: D\ot P\lra P\ot D, \qquad d \ot p\mapsto (\id \ot \pi)\circ \psi(c\ot p), \quad c\in \pi^{-1}(d),
$$
is bijective.
\end{blist}
The assumption (a) can be made with no harm, and indeed without it one could hardly talk about reflecting classical inclusion of a subgroup into the structure group. The assumption (b) is always satisfied if $D$ has a cointegral (i.e.\ it is a coseparable coalgebra), which is the case for instance if $D$ is a coalgebra of a compact quantum group (the cointegral is obtained from the Haar measure). There seem to be no natural geometrically motivated condition that would force (c) to hold in generality of Theorem~\ref{thm.main}. On the other hand, if the existence of additional symmetry is requested as in Theorem~\ref{thm.main.DK}, then property (c) follows in view of the explicit form of $\psi$ as described in Remark~\ref{rem.DK} and the fact that $\pi$ is a right $H$-module homomorphism.
 \end{rema}

\section{Examples}\label{sec.ex}
In order to capture more fully geometric contents of the examples as well as to make forays into the topological world of $C^*$-algebras of continuous functions on noncommutative topological spaces, we assume in this section that $\mathbb{K}=\C$.

\subsection{The quantum flag manifold}\label{sec.flag}
The first example presents the quantum flag manifold as a bundle with the quantum homogeneous fibre, and it illustrates a special case of Theorem~\ref{thm.main.DK} in which $H=C$. We set
$$
H = C =\pol(U_q(2))
$$
to be  a quantum group of matrix type with $*$-algebra generators $u,\alpha,\gamma$ organised into a unitary matrix
\begin{equation}\label{uq2}
\mathbf{v} =  \begin{pmatrix} u & 0 & 0 \cr
0 & \alpha & -q\gamma^*u^*\cr 
0 & \gamma & \alpha^*u^*
\end{pmatrix}.
\end{equation}
These satisfy the relations: $u$ is a central unitary,  while
\begin{subequations}\label{su2}
\begin{equation}
\label{su2.1}
\alpha\gamma = q \gamma\alpha,  \qquad \gamma\gamma^* = \gamma^*\gamma ,
\end{equation}
\begin{equation}\label{su2.2}
\alpha\gamma^* = q \gamma^*\alpha, \qquad \alpha^*\alpha + \gamma\gamma^* = 1, \qquad \alpha\alpha^* + q^2 \gamma\gamma^*=1;
\end{equation}
\end{subequations}
see \cite{w1}.  Since $\pol(U_q(2))$ is a $*$-Hopf algebra of matrix type, the comultiplication and counit come out as
\begin{equation}\label{uq2.copr}
\Delta(u) = u\ot u, \quad \Delta(\alpha) = \alpha\ot \alpha -q\gamma^*u^*\ot \gamma,\quad \Delta(\gamma) = \gamma\ot\alpha +\alpha^*u^*\ot\gamma.
\end{equation}
\begin{equation}\label{uq2.counit}
\eps(u) = \eps(\alpha) =1, \qquad \eps(\gamma) = 0.
\end{equation}

Let $0\leq s \leq 1 $, and set
\begin{equation}\label{xz}
\xi = (1-s^2)\gamma\gamma^* + s(\gamma\alpha u+\alpha^*\gamma^*{u^*}),  \quad 
\zeta = (1-s^2)\alpha\gamma^* + s(\alpha^2u-q{\gamma^*}^2{u^*}),
\end{equation}
 Clearly, $\xi^*=\xi$, and the elements  $\xi$,  $\zeta$ and $\zeta^*$ satisfy  relations
\begin{equation}\label{rel.all.sphere}
\zeta\xi = q^2\xi\zeta, \quad \zeta\zeta^* = (s^2+q^2\xi)(1-q^2\xi), \quad \zeta^*\zeta = (s^2+\xi)(1-\xi),
\end{equation}
and thus they generate a $*$-subalgebra of $\pol(U_q(2))$ isomorphic with the coordinate algebra of the (generic) Podle\'s sphere $\pol(S^2_{q,s})$ \cite{Pod:sph} or a two-parameter deformation of the complex projective line $\pol(\C P^1_{q,s})$. We identify the $*$-algebra generated by $\xi,\zeta$ with $\pol(\C P^1_{q,s})$. Using the explicit form of $\Delta$ in \eqref{uq2.copr} one can check that $\pol(\C P^1_{q,s})$ is a left coideal subalgebra of $\pol(U_q(2))$. As explained e.g.\ in \cite{Brz:hom} or \cite{MulSch:hom}, $\pol(\C P^1_{q,s})$ can be also identified with fixed points of the coaction of a specific coalgebra and right $\pol(U_q(2))$-module. Consider the right ideal $J$ generated by the restriction of $\pol(\C P^1_{q,s})$ to the kernel of the counit of $\pol(U_q(2))$, that is
\begin{equation}\label{ideal}
J = \langle \xi, \zeta-s, \zeta^*-s\rangle \pol(U_q(2))\subset \pol(U_q(2)).
\end{equation}
One can easily check (or invoke \cite{Brz:hom}) that $J$ is a coideal in $\pol(U_q(2))$ and thus we have a coalgebra and right $\pol(U_q(2))$-module epimorphism
\begin{equation}\label{uq2.pi}
\pi: \pol(U_q(2)) \lra D:= \pol(U_q(2))/J.
\end{equation}
The results of \cite{MulSch:hom} ensure that $\pol(\C P^1_{q,s})$ are coinvariants of the induced coaction,
\begin{equation}\label{uq2.coinv}
\pol(\C P^1_{q,s}) = \pol(U_q(2))^\coD_{\pi(1)} = \{x\in \pol(U_q(2))\; |\; (\id\ot \pi)\circ\Delta(x) = x\ot \pi(1)\},
\end{equation}
and that $D$ is spanned by group-like elements. The form of these elements can be worked out as in \cite[Section~6]{BrzMaj:geo}.

\begin{prop}\label{prop.basis}
A basis for $D$ is formed of the following group-like elements, for all $m,n\in \Z$: 
\begin{equation}
d_{m,n} =  
\begin{cases}
 \pi\left(u^m{\displaystyle\prod_{k=0}^{n}} \left(\alpha - q^{k}s\gamma^*u^*\right)\right) = \pi\left(u^m{\displaystyle\prod_{k=0}^{n-1}}\left(\alpha + q^{k}s\gamma\right)\right) , &  n>0, \cr
\cr
\pi(u^m), & n=0,\cr \cr
\pi\left(u^m{\displaystyle\prod_{k=1}^{-n}} \left(\alpha^* +q^{-k}s\gamma^*\right)\right) = \pi\left(u^m{\displaystyle\prod_{k=0}^{-n-1}} \left(\alpha^* - q^{-k}s\gamma u \right)\right) , & n<0, 
\end{cases}
\end{equation}
where the products are from left to right.
\end{prop}
\begin{proof}
Clearly $u^n$ are group-like elements. Next observe that, for all $a\in \pol(U_q(2))$, $\pi(\xi a) =0$ and $\pi(\zeta a) = \pi(\zeta^* a) =s\pi(a)$, and thus we can compute
\[
\begin{aligned}
0 = & \pi(\xi \alpha^*) = \pi\left((1-s^2)\gamma\gamma^*\alpha^* + s\gamma\alpha \alpha^* u+\alpha^*\gamma^*\alpha^*{u^*}\right)\\
&= \pi\left(\left((1-s^2)q\gamma\alpha^* + sq\alpha^{*2}u-q^2{\gamma}^2{u^*}\right)\gamma^*\right) +s\pi(\gamma u) \\
&= q\pi(\zeta^*\gamma^*) +s\pi(\gamma u) = qs\pi(\gamma^*) + s\pi(\gamma u).
\end{aligned}
\]
Therefore, for all $a\in \pol(U_q(2))$,
\begin{equation}\label{gg*}
\pi(\gamma a) = -q\pi(\gamma^*u^* a), \qquad \pi(\gamma^* a) = -q^{-1}(\gamma u a),
\end{equation}
where the second equality follows by the centrality and unitarity of $u$. Using these relations, while taking care that $\pi$ is a right $\pol(U_q(2))$-module homomorphism (not an algebra map!), one can inductively prove the equalities in the product formulae for $d_{m,n}$. 

In a similar way, the equality $\pi((\zeta-s)\alpha^*)$ implies that, for all $a\in  \pol(U_q(2))$,
\begin{equation}\label{ag1}
\pi\left(\left(s\alpha^* u^*-q\gamma^*u^*\right)a\right) = s\pi\left(\left(\alpha -sq\gamma^*u^*\right)a\right),
\end{equation}
while $\pi((\zeta^*-s)\alpha)$ implies
\begin{equation}\label{ag2}
\pi\left(\left(s\alpha-\gamma\right)a\right) = s\pi\left(\left(\alpha^*u^* -s\gamma\right)a\right).
\end{equation}

The final technical equality arises from the analysis of the commutation rules in $\pol(U_q(2))$. These imply, for all $m,n\in \Z$ and $t\in \C$:
\begin{equation}\label{com.n}
\left(\alpha -q^nt\gamma^*u^m\right)\left(t\alpha^*u^m+q^{-n}\gamma\right) = \left(t\alpha^*u^m +q^{-n+1}\gamma\right)\left(\alpha - q^{n+1}t\gamma^*u^m\right).
\end{equation}

Armoured with these equalities we can proceed to prove that $d_{m,n}$ are group-like elements of $D$. Obviously, all the $d_{m,0}$ are group-like. Since $u$ is group-like and central, and $\pi$ is a coalgebra and right $\pol(U_q(2))$-module homomorphism suffices it to prove that the elements  $d_{0,n}$ are group-like.
We observe first that, for all $n\in \N$,
\begin{subequations}\label{d.act}
\begin{equation}\label{d.act+}
sd_{0,n+1}  = d_{0,n}\cdot (s\alpha^*u^* +q^{-n}\gamma)= d_{0,n}\cdot (s\alpha^*u^* -q^{-n+1}\gamma^*u^*) ,
\end{equation}
\begin{equation}\label{d.act-}
sd_{0,-n-1} = d_{0,-n}\cdot (s\alpha u - q^n\gamma u) = d_{0,-n}\cdot (s\alpha u+ q^{n+1}\gamma^*)
\end{equation}
\end{subequations}
The recursive relations \eqref{d.act} can be proven by induction. Indeed,
$$
\begin{aligned}
d_{0,1}\cdot (s\alpha^*u^* +q^{-1}\gamma)&= \pi\left(\left(\alpha-qs\gamma^*u^*\right)\left(s\alpha^*u^* +q^{-1}\gamma\right)\right)\\
&=\pi\left(\left(s\alpha^*u^* +\gamma\right)\left(\alpha - q^{2}s\gamma^*u^*\right)\right)\\
&=s\pi\left(\left(\alpha-qs\gamma^*u^*\right)\left(\alpha - q^{2}s\gamma^*u^*\right)\right) = d_{0,2},
\end{aligned}
$$
where \eqref{com.n} (with $m=-1$, $n=1$ and $t=s$) was used to derive the second equality and \eqref{ag1} to derive the third one. Next, assume that the equality \eqref{d.act+} holds for $n-1$. Then using the definition of $d_{m,n}$ and the equality  \eqref{com.n} we obtain
$$
\begin{aligned}
d_{0,n}\cdot (s\alpha^*u^* +q^{-n}\gamma)&= d_{0,n-1}\cdot \left(\left(\alpha-q^ns\gamma^*u^*\right)\left(s\alpha^*u^* +q^{-n}\gamma\right)\right)\\
&=d_{0,n-1}\cdot \left(\left(s\alpha^*u^* +q^{-n+1}\gamma\right) \left(\alpha - q^{n+1}s\gamma^*u^*\right)\right)\\
&=\left(d_{0,n-1}\cdot \left(s\alpha^*u^* +q^{-n+1}\gamma\right)\right)\cdot \left(\alpha - q^{n+1}s\gamma^*u^*\right)\\
&=sd_{0,n}\cdot \left(\alpha - q^{n+1}s\gamma^*u^*\right) = d_{0,n+1},
\end{aligned}
$$
where the fourth equality follows by the inductive assumption. This proves the first of equalities in \eqref{d.act+}. The second one follows by \eqref{gg*}. The equalities \eqref{d.act-} are proven in a similar way with the help of \eqref{ag2}. 

We are now in position to prove that the $d_{0,n}$ are group-like. We deal only with the positive $n$ case, the other case is proven in a similar way. First we use \eqref{uq2.copr} and the definition of $\pi$ to compute
$$
\begin{aligned}
\Delta_D(d_{0,1}) &= (\pi\ot\pi)\circ\Delta (\alpha +s\gamma) \\
&= \pi(\alpha +s\gamma)\ot \pi(\alpha)  + \pi(s\alpha^*u^*-q\gamma^*u^*)\ot\pi(\gamma)\\
&= d_{0,1}\ot \pi(\alpha) + s \pi(\alpha-q\gamma^*u^*)\ot\pi(\gamma)\\
&= d_{0,1}\ot \pi(\alpha) + s \pi(\gamma)) = d_{0,1}\ot d_{0,1},
\end{aligned}
$$
where the third equality follows by \eqref{ag1}. Thus $d_{0,1}$ is a group-like element. Now, assume that $d_{0,n}$ is group-like, then first using the recursive definition of $d_{0,n}$  and then relations \eqref{d.act+} we can compute
\[
\begin{aligned}
\Delta_D(d_{0,n+1}) &= \Delta_D\left(d_{0,n}\cdot (\alpha +q^n\gamma)\right) \\
&= d_{0,n}\cdot \left(\alpha +q^ns\gamma\right) \ot d_{0,n}\cdot \alpha +
q^n d_{0,n}\cdot \left(s\alpha^*u^* -q^{-n+1}\gamma^*u^*\right) \ot d_{0,n}\cdot \gamma \\
&= d_{0,n+1}\ot d_{0,n}\cdot \alpha +
sq^n d_{0,n+1}\ot d_{0,n}\cdot \gamma = d_{0,n+1}\ot d_{0,n+1},
\end{aligned}
\]
and thus conclude that all the $d_{0,n}$, $n\in \N$, and hence $d_{m,n}$, $m\in\Z$, $n\in\N$ are group-like. The negative $n$ case is proven in a similar way. Finally, that the $d_{m,n}$ span $D$ can be proven as in \cite[Proposition~6.1]{Brz:hom}.
\end{proof}

Since the coalgebra $D$ is spanned by group-like elements labelled by elements of $\Z\times \Z$ it can be equipped with a Hopf algebra structure of the group algebra of the free Abelian group $\Z\times \Z$, i.e.\  with the product 
$$
d_{m,n}d_{k,l} = d_{m+k,n+l}, \qquad \mbox{for all $k.l.m.n\in\Z$}.
$$
This Hopf algebra can be identified with the algebra of Laurent polynomials in two variables, say $u_1$ and $u_2$, by $d_{m,n} = u_1^mu_2^n$ and thus simply with $\pol(\T^2)$.

For the algebra $P$ we take the coordinate algebra of the quantum group $SU_q(3)$, i.e.\ $P= \pol(SU_q(3))$. This is  a $*$-Hopf algebra of matrix group type,  generated by  the  entries of the quantum matrix $\mathbf{u} = (u_{ij})_{i,j=1}^3$, which satisfy the following non-commutation relations:
\begin{subequations}\label{qmatrix}
\begin{eqnarray}
u_{ij}u_{ik} & = & qu_{ik}u_{ij}, \;\;\;\;\; j<k, \label{qmatrix1} \\
u_{ji}u_{ki} & = & qu_{ki}u_{ji}, \;\;\;\;\; j<k, \label{qmatrix2} \\
u_{ij}u_{km} & = & u_{km}u_{ij}, \;\;\;\;\; i<k, \; j>m, \label{qmatrix3} \\
u_{ij}u_{km}-u_{km}u_{ij} & = & (q-q^{-1})u_{im}u_{kj},
\;\;\;\;\; i<k, \; j<m, \label{qmatrix4}
\end{eqnarray}
\end{subequations}
with $i,j,k,m\in\{1,2,3\}$, together with the  determinant relation 
\begin{equation}\label{determinant}
\sum_{i_1=1}^3\sum_{i_2=1}^3\sum_{i_3=1}^3 E_{i_1i_2i_3}
u_{j_1 i_1}u_{j_2 i_2}u_{j_3 i_3}=E_{j_1 j_2 j_3},
\;\;\;\;\; \forall(j_1,j_2,j_3)\in\{1,2,3\},
\end{equation}
where
\begin{equation}\label{inversions}
E_{i_1 i_2 i_3}=\begin{cases} (-q)^{I(i_1,i_2,i_3)} & \text{if}
\; i_r\neq i_s \; \text{for} \; r\neq s, \\ 0 & \text{otherwise,}
\end{cases}
\end{equation}
and $I(i_1,i_2,i_3)$ denotes the number of inversed pairs in the
sequence $i_1,i_2,i_3$; \cite{d,so,frt,ks}. The $*$-structure is given by 
\begin{equation}\label{u*}
u_{ij}^* = (-q)^{j-i}\left(u_{i_1j_1}u_{i_2j_2} - q u_{i_1j_2}u_{i_1j_1}\right),
\end{equation}
where $i_1<i_2\in \{1,2,3\}\setminus \{i\}$ and $j_1<j_2\in \{1,2,3\}\setminus \{j\}$. Finally, the comultiplication, counit  and the antipode are:
$$ 
\Delta(u_{ij})=\sum_{k=1}^n u_{ik}\otimes u_{kj}, \qquad \eps(u_{ij}) = \delta_{ij}, \qquad S(u_{ij}) = u_{ji}^*. 
$$
As observed by Br\c{a}giel \cite{bragiel}, $\pol(SU_q(3))$ is a dense $*$-subalgebra of the $C^*$-algebra $C(SU_q(3))$ of the continuous functions on the compact quantum group $SU_q(3)$ introduced by Woronowicz in \cite{w2}, \cite{w3}. 

The map determined by $\mathbf{u}\lto \mathbf{v}$ is a $*$-Hopf algebra epimorphism $\pol(SU_q(3))\lra \pol(U_q(2))$, which makes $\pol(SU_q(3))$ into a right $\pol(U_q(2))$-comodule algebra with the coaction
$$
\varrho = \delta: \pol(SU_q(3))\lra \pol(SU_q(3))\ot \pol(U_q(2)), \qquad u_{ij}\lto \sum_{k=1}^n u_{ik}\otimes v_{kj}.
$$
As shown in \cite{BrzSzy:fla}, $\pol(SU_q(3))$ is  a principal $\pol(U_q(2))$-comodule algebra with coinvariants $\pol(\C P_q^2)$. Bearing in mind the (possible) interpretation of $D$ as the coalgebra part of the coordinate algebra $\pol(\T^2)$ we can interpret the fixed point subalgebra of $\pol(SU_q(3))$ under the coaction $(\id\otimes \pi)\circ\varrho$ as the algebra of coordinate functions on the quantum flag manifold $SU_q(3)/\T^2$. As the definition of the $\pi$  depends on an additional parameter $s$ rather than obtaining the standard quantum flag manifold such as those considered in \cite{StoDij:fla} we obtain its more general, two-parameter version, which we denote by $FM_{q,s}$. Thus the coordinate algebra of $FM_{q,s}$ is given by
$$
\pol(FM_{q,s}) = \{ x\in \pol(SU_q(3)) \;|\; (\id\otimes \pi)\circ\varrho (x) = x\ot d_{0,0}\}.
$$
Since we fulfil all the assumptions of Theorem~\ref{thm.main.DK} we thus conclude that 
$$
\pol(FM_{q,s}) \cong \pol(SU_q(3))\,\Box_{\pol(U_q(2))} \pol(\C P_{q,s}^1),
$$
and that 
$$
\pol(FM_{q,s})\ot_{\pol(\C P_q^2)}\pol(FM_{q,s})\cong \pol(SU_q(3))\,\Box_{\pol(\T^2)} \pol(\C P_{q,s}^1),
$$
i.e.\ $FM_{q,s}$ is a noncommutative bundle over $\C P_q^2$ with fibre $\C P_{q,s}^1$. All this can be summarised by the following diagram:
$$
\xymatrix{H= \pol(U_q(2))\ar@{=}[d] & & P = \pol(SU_q(3)) \ar@{--}[lld]_{\mbox{\small(principal)}}  \\
C=\pol(U_q(2))\ar@{->>}[d]_{\mbox{\small (coalg.\ r.\ $H$-lin.)}}^\pi &  X = \pol(\C P_{q,s}^1)\ar@{_{(}->}[l] & A = \pol(FM_{q,s})\ar@{_{(}->}[u]  \\
D=\pol(\T^2) & 
& B = \pol(\C P_q^2)\ar@{_{(}->}[u] \ar@{--}[llu]^{\mbox{\small(extension)}}}
$$
\medskip

\begin{rema}
It is worth noting that quantum flag manifolds in general and the flag manifold of $SU_q(3)$ in particular have been quite extensively studied already. For example, aspects of the theory related to: Dirac operators (\cite{k}), function algebras (\cite{NevTus:hom}), Fredholm modules (\cite{vy}), 
complex structures (\cite{OB}, \cite{Mata}), 
have been thoroughly investigated. The novelty of the above example from the present paper is a clear-cut interpretation of the quantum flag manifold 
in question as a noncommutative bundle. Another novelty is appearence of generic Podle\'{s} spheres as fibres. 
\end{rema}

\subsection{The quantum twistor bundle}\label{sec.twistor}

The aim of this section is to show that a quantum twistor bundle,
$$  
\pol(\C P_q^1) \lra \pol(\C P_q^3) \lra \pol(S_q^4),
 $$
also fits into the algebraic framework described in Theorems \ref{thm.main} and \ref{thm.main.DK} above. 
The definition was given  in \cite{MSz}. It takes as its point of departure the noncommutative instanton bundle 
$$  \pol(SU_q(2)) \lra \pol(S_q^7) \lra \pol(S_q^4) $$
constructed in \cite{BCT} and analysed further in \cite{BCDT}.  $C^*$-algebraic aspects are discussed in \cite{MSz}. In the present paper, 
we deal with purely algebraic aspects only. 

Let $\pol(U_q(4))$ be the polynomial algebra of the quantum unitary group, with $q\in (0,1)$. This is a universal $\C$-algebra generated by 
elements $\{t_{ij}\}_{i,j=1}^4$ and $D_q^{-1}$, subject to the following relations: 
$$ \begin{aligned}
& t_{ik}t_{jk} = qt_{jk}t_{ik}, \;\; t_{ki}t_{kj} = qt_{kj}t_{ki}, \;\; i<j,  \\
& t_{il}t_{jk} = t_{jk}t_{il}, \;\; i<j, k<l, \\
& t_{ik}t_{jl} - t_{jl}t_{ik} = (q-q^{-1})t_{jk}t_{il}, \;\;  i<j, k<l, \\
& D_qD_q^{-1} = D_q^{-1}D_q =1. 
\end{aligned} $$
Here $D_q$ is the quantum determinant, defined as 
$$ D_q = \sum_{\sigma\in S_4}(-q)^{I(\sigma)}t_{\sigma(1)1}\cdots t_{\sigma(4)4}, $$
where $I(\sigma)$ denotes the number of inversed pairs and $S_4$ is the symmetric group on four letters. This algebra equipped with the usual 
comultiplication $\Delta_{U_q(4)}$, counit $\varepsilon$, and antipode  $S_{U_q(4)}$,
$$ \begin{aligned}
\Delta_{U_q(4)}(t_{ij}) & = \sum_{k} t_{ik} \otimes t_{kj}, \\
\varepsilon(t_{ij}) & = \delta_{ij}, \\
S_{U_q(4)}(t_{ij}) & = (-q)^{i-j}\sum_{\sigma\in S_3}(-q)^{I(\sigma)}t_{j_{\sigma(1)}i_1}t_{j_{\sigma(2)}i_2}t_{j_{\sigma(3)}i_3}. 
\end{aligned} 
$$
 is a Hopf algebra (see e.g.\ \cite[p. 311--314]{ks}).
Here $\{j_1,j_2,j_3\}=\{1,\ldots,4\}\setminus\{j\}$  and $\{i_1,i_2,i_3\}=\{1,\ldots,4\}\setminus\{i\}$. In fact
$\pol(U_q(4))$ is a Hopf $*$-algebra with the involution
$$ t_{ij}^* = S_{U_q(4)}(t_{ji}) \;\; \text{and} \;\; D_q^*=D_q^{-1}. $$

The $*$-subalgebra of $\pol(U_q(4))$ generated by $\{z_i=t_{4i} \mid i=1,\ldots,4\}$ plays the role of a polynomial algebra of the quantum $7$-sphere. 
In fact, $\pol(S_q^7)$ is the universal $*$-algebra for the following relations:
$$ \begin{aligned}
& z_iz_j = qz_jz_i, \;\; i<j, \;\; \text{and} \;\; z_j^*z_i = qz_iz_j^*, \;\; i\neq j, \\
& z_k^*z_k = z_kz_k^* + (1-q^2)\sum_{j<k}z_jz_j^*, \;\; \text{and} \;\; \sum_{k=1}^4 z_kz_k^* = 1. 
\end{aligned} $$
The enveloping $C^*$-algebra $C(S_q^7)$ coincides with the $C^*$-algebra of continuous functions on the Vaksman-Soibelman quantum 7-sphere 
\cite{vs} and in particular it is a graph $C^*$-algebra \cite{hs}. We note that the generator $z_i$ from \cite{BCT}, \cite{BCDT} (and from the present paper) 
corresponds to $z_{5-i}$ from \cite{hs}.  

Now, let $M$ be the right ideal of  $\pol(U_q(4))$ generated by 
$$ \{ t_{13}, t_{31}, t_{14}, t_{41}, t_{24}, t_{42}, t_{23}, t_{32}, t_{11}-t_{44}, t_{22}-t_{33}, t_{12}+t_{43}, t_{21}+t_{34}, 
t_{11}t_{22}-qt_{12}t_{21}-1\}. $$
Let $\pi:\pol(U_q(4)) \rightarrow \pol(U_q(4))/M$ be the corresponding quotient map. If one identifies $\pi(t_{ij})$, $i,j=1,2$, with the elements 
of the fundamental matrix of $SU_q(2)$, 
$$ \left(\hspace{-1mm} \begin{array}{cc} \alpha & -q\gamma^* \\ \gamma & \alpha^* \end{array} \hspace{-1mm}\right), $$
then $\pi$ becomes a right coalgebra map and a right $\pol(U_q(4))$-module map, in the sense that 
$\pi(a)=\pi(b)$ implies $\pi(ac)=\pi(bc)$, for all $a,b,c\in\pol(U_q(4))$. In what follows, we denote this map  by $\pi_{SU_q(2)}$.  Clearly, 
$$ 
\varrho_{U_q(4)} = (\id\otimes\pi_{SU_q(2)})\circ \Delta_{U_q(4)} 
$$ 
puts on $\pol(U_q(4))$ a right $\pi_{SU_q(2)}$-comodule structure. Its restriction to $\pol(S_q^7)$ 
yields a right $\pi_{SU_q(2)}$-comodule structure on $\pol(S_q^7)$ via 
$$ \varrho_{S_q^7}:\pol(S_q^7) \longrightarrow \pol(S_q^7) \otimes \pol(SU_q(2)). $$
The extension $\pol(S_q^7)\subseteq\pol(U_q^4)$ is copointed, since $\varrho_{S_q^7}(1)=1\otimes 1$. We have that $\pol(S_q^7)^{\mathrm{co}\,SU_q(2)}=
\pol(S_q^7)_1^{\mathrm{co}\,SU_q(2)}$ is a $*$-subalgebra of $\pol(S_q^7)$, denoted by $\pol(S_q^4)$. This $*$-algebra is generated by 
elements 
$$ a=z_1z_4^*-z_2z_3^*, \;\; b=z_1z_3+q^{-1}z_2z_4, \;\; R=z_1z_1^*+z_2z_2^*, $$
satisfying the following  relations
$$ \begin{aligned}
& Ra = q^{-1}aR, \;\; Rb = q^2bR, \;\; ab = q^3ba, \;\; ab^* = q^{-1}b^*a, \\
& aa^* - q^2a^*a = (1-q^2)R^2, \\
& b^*b - q^4bb^* = (1-q^2)R, \\
& aa^* + q^2bb^* = R(1-q^2R). 
\end{aligned} $$ 
$\pol(S_q^4)$ is thought of as the polynomial algebra of a quantum 4-sphere. Its enveloping $C^*$-algebra is isomorphic to the minimal unitization 
of the compacts. 

It is shown in \cite{BCDT} that the canonical Galois map 
$$ \can:\pol(S_q^7)\otimes_{\pol(S_q^4)} \pol(S_q^7) \longrightarrow \pol(S_q^7) \otimes \pol(SU_q(2)) $$ 
is bijective, and so is the entwining map 
$$ \psi:\pol(SU_q(2))\otimes \pol(S_q^7) \longrightarrow \pol(S_q^7) \otimes \pol(SU_q(2)).  $$ 
Furthermore, since the Galois map is bijective and coalgebra $SU_q(2)$ is coseparable (i.e. it is equipped with a Haar measure), it follows from 
\cite[Theorem 4.6]{Brz:comod} that $\pol(S_q^7)$ is $SU_q(2)$-equivariantly projective as a left $\pol(S_q^4)$-module. 

Now, let $u$ be the standard unitary generator of $\pol(U(1))$, and let $\pi_{U(1)}:\pol(SU_q(2)) \longrightarrow \pol(U(1))$
be the Hopf $*$-algebra surjection  such that $\pi_{U(1)}(\alpha)=u$ and $\pi_{U(1)}(\gamma)=0$. 
Then the map 
$$ \bar{\varrho}_{S_q^7} : \pol(S_q^7) \longrightarrow \pol(S_q^7) \otimes \pol(U(1)), \;\;\;\;  
 \bar{\varrho}_{S_q^7} = (\id \otimes \pi_{U(1)}) \circ \varrho_{S_q^7}, $$
yields a right $U(1)$-comodule structure on $\pol(S_q^7)$. We claim that in fact $\bar{\varrho}_{S_q^7}$ is a $*$-homomorphism. Indeed, let 
$\mu:\pol(U_q(4))\longrightarrow\pol(U(1))$ be a $*$-homomorphism such that 
$$ \mu(t_{ij}) = \left\{ \hspace{-2mm}\begin{array}{lll} u & \text{if} & i=j\in\{1,4\} \\ u^* & \text{if} & i=j\in\{2,3\} \\ 0 & \text{if} 
& i\neq j \end{array} \right. $$
Let $E$ be the subalgebra of $\pol(U_q(4))$ generated by $t_{ij}$, $i,j=1,2$. Then, for each $x\in\pol(U_q^4)$ there are $y\in E$ and $m\in M$ 
such that $x=y+m$. Since $M\subseteq\ker(\mu)$ and maps $\pi_{u(1)}\circ \pi_{SU_q(2)}$ and $\mu$ coincide on $E$, we have 
$$
\pi_{u(1)}\circ \pi_{SU_q(2)}(x) = \pi_{u(1)}\circ \pi_{SU_q(2)}(y) = \mu(y) = \mu(x).
$$
 Thus $\pi_{u(1)}\circ \pi_{SU_q(2)}$ is a $*$-homomorphism and, 
consequently, so is $\bar{\varrho}_{S_q^7}$. This in turn implies that $\pol(S_q^7)_1^{\mathrm{co}\,\pol(U(1))} = \pol(S_q^7)^{\mathrm{co}\,\pol(U(1))}$ is a $*$-subalgebra 
of $\pol(S_q^7)$ containing $\pol(S_q^4)$. 

According to the framework presented in Section~\ref{sec.results}, the fibre of the twistor bundle is defined as $\pol(SU_q(2))_1^{\mathrm{co}\,\pol(U(1))}$. It coincides with the $*$-subalgebra 
$\pol(SU_q(2))^{\mathrm{co}\,\pol(U(1))}$ of $\pol(SU_q(2))$, and it is simply the polynomial algebra $\pol(\C P_q^1)$ of the quantum complex projective 1-space 
(the standard Podle\'{s} sphere). 

In view of the above discussion, we see that all the assumptions appearing in Theorems \ref{thm.main} and \ref{thm.main.DK} are fulfilled. The situation described in this section can thus be summarised in the following diagram
$$
\xymatrix{H= \pol(U_q(4))\ar@{->>}[d]_{\mbox{\small(coalg.\ r.\ $H$-lin.)}}^{\pi_{SU_q(2)}} & & P = \pol(S_q^7)\ar@{_{(}->}[ll]_{\mbox{\small(4-th row)}} \ar@{--}[lld]_{\mbox{\small(principal)}}  \\
C=\pol(SU_q(2))\ar@{->>}[d]_{\mbox{\small (Hopf alg.\ map)}}^{\pi_{U(1)}} &  X = \pol(\C P_q^1)\ar@{_{(}->}[l] & A = \pol(\C P_q^3)\ar@{_{(}->}[u]  \\
D=\pol(U(1)) & 
& B = \pol(S_q^4)\ar@{_{(}->}[u] \ar@{--}[llu]^{\mbox{\small(extension)}}}
$$

\begin{rema}
In the example above, in order to construct the twistor bundle we took as point of departure the quantum instanton bundle defined and investigated in 
\cite{BCT} and \cite{BCDT}. In this way we could illustrate the full power of our Theorem  \ref{thm.main.DK}. However, one should note that quite 
different constructions of quantum instanton bundles exist in the literature, for example those from \cite{lpr} and \cite{bm}. 
\end{rema}

\section{The framework: proofs}\label{sec.proofs}
In this section we gather all technical data and definitions and present the proofs of the main results. Throughout this section $C$ generally denotes a coalgebra with comultiplication $\Delta: C\lra C\otimes C$ and counit $\varepsilon$, and $P$ is an algebra with multiplication $\mu: P\otimes P\lra P$ and identity map $\eta: \k \lra P$; the identity element is denoted by $1$. Often $P$ is also a right $C$-comodule with coaction $\varrho: P\lra P\otimes C$. We use Sweedler's notation for comultiplication and coaction,
$$
\Delta(c) = \sum c\sw 1\otimes c\sw 2, \qquad \varrho(p) = \sum p\sw 0\otimes p\sw 1, \qquad \mbox{for all $c\in C$, $p\in P$}.
$$
If $P$ is a right $C$-comodule and $X$ is a left $C$-comodule, then $P\square_C X$ denotes the cotensor product; see \eqref{coten.pcx}.

An {\em entwining structure} is a triple $(P,C,\psi)$, where $P$ is an algebra $C$ is a coalgebra and $\psi: P\otimes C\lra C\otimes P$ is a linear map that satisfies the following conditions \cite{BrzMaj:coa}:
\begin{subequations}\label{entw}
\begin{equation}\label{ent.al}
\psi \circ (\id_C \otimes \mu) = (\mu \otimes \id_C)\circ (\id_P\otimes \psi)\circ (\psi \otimes \id_P), \qquad \psi\circ (\id_C\otimes \eta) = \eta \otimes\id_C.
\end{equation}
\begin{equation}\label{ent.co}
(\id_P \otimes \Delta)\circ \psi  =(\psi \otimes \id_C) \circ (\id_C\otimes \psi)\circ (\Delta \otimes \id_P) , \qquad (\id_P\otimes \eps) \circ \psi  = \eps \otimes\id_P,
\end{equation}
\end{subequations}
To denote the action of an entwining map $\psi$ on elements of $P$ and $C$ we use the notation
\begin{equation}\label{alpha}
\psi(c\ot p) = \sum_\alpha p_\alpha\ot c^\alpha.
\end{equation}
In terms of this notation the conditions \eqref{entw} come out as, for all $p,p'\in P$ and $c\in C$,
$$
\begin{aligned}
& \sum_\alpha (pp')_\alpha \ot c ^\alpha = \sum_{\alpha,\beta} p_\alpha p'_\beta \ot c^{\alpha\beta}, \qquad \sum_\alpha 1_\alpha \ot c^\alpha = 1\ot c,\\
&\sum_\alpha p_\alpha \ot c^\alpha\sw 1\ot c^\alpha\sw 2 = \sum_{\alpha,\beta} p_{\alpha\beta} \ot c^\beta\sw 1\ot c^\alpha\sw 2,  \qquad \sum_\alpha p_\alpha\eps(c^\alpha) = p\eps(c).
\end{aligned}
$$
If the entwining map $\psi$ is bijective, then the action of $\psi^{-1}$ on elements of $P$ and $C$  is denoted by
\begin{equation}\label{inv.alpha}
\psi^{-1} (p\ot c) = \sum_A c_A\ot p^A.
\end{equation}
Equations \eqref{entw} yield the following identities for the inverse of an entwining map:
\begin{subequations}\label{entw.inv}
\begin{equation}\label{entw.inv.al}
\sum_A  c _A \ot (pq)^A= \sum_{A,B} c_{AB}\ot p^B q^A, \qquad \sum_A c_A \ot 1^A = c\ot 1,
\end{equation}
\begin{equation}\label{entw.inv.co}
\sum_A c_A\sw 1\ot c_A\sw 2\ot p^A  = \sum_{A,B}  c_A\sw 1\ot c_B\sw 2\ot p^{AB},  \qquad \sum_A p^A\eps(c_A) = p\eps(c),
\end{equation} 
\end{subequations}
for all $p,q\in P$ and $c\in C$. Furthermore, the combinations of \eqref{alpha} and \eqref{inv.alpha} yield, for all $c\in C$ and $p\in P$,
\begin{equation}\label{inv}
\sum_{\alpha,A} c^\alpha{}_A \ot p_\alpha{}^A =c\ot p, \qquad \sum_{\alpha,A} p^A{}_\alpha \ot c_A{}^\alpha = p\ot c.
\end{equation}
The main examples of an entwining map are the canonical entwining map associated to a coalgebra-Galois extension \eqref{can.entw} and the Doi-Koppinen entwining in Remark~\ref{rem.DK}.

Let $(P,C,\psi)$ be an entwining structure. An {\em entwined module} is a right $P$-module and a right $C$-comodule $M$ such that, for all $m\in M$ and $p\in P$,
$$
\sum (m\cdot p)\sw 0 \ot (m\cdot p)\sw 1 =  \sum_\alpha m\sw 0\cdot p_\alpha \ot m\sw 1^\alpha.
$$
The canonical entwining structure associated to a coalgebra-Galois $C$-extension $P$ \eqref{can.entw}  is the unique one, for which $P$ is an entwined module with the $P$-action given by multiplication and the $C$-coaction $\varrho: P\to P\ot C$. In general, if $(P,C,\psi)$ is an entwining structure and $P$ is an entwined module by multiplication and such that $\varrho(1) = 1\ot e$, for a (necessarily) group-like element $e\in C$, then $\varrho(p) = \psi(e\ot p)$, for all $p\in P$, and 
$
P_e^\coC= P^\coC$ (see Definition~\ref{def.coal.Gal} and equation \eqref{e-coinv}), so, in particular, the $e$-coinvariants form a subalgebra of $B$.

We make the following standing assumptions. $C$ and $D$ are coalgebras,  and $\pi: C\lra D$ is a coalgebra map. We consider $C$ as a right $D$-comodule by $(\id \ot \pi)\circ \Delta_C$. We fix a group-like element $e\in C$, set $\bar e = \pi(e)\in D$ and define $X = C_{\bar e}^{\coD}$ (see \eqref{x}). Furthermore, $P$ is a right $C$-comodule with coaction $\varrho: P\to P\ot C$. We view $P$ as a $D$-comodule by $\bar\varrho = (\id\ot \pi)\circ \varrho$ and set $A = P_{\bar e}^{\coD}$ (see \eqref{coinv.be}) and $B = P^\coC$ (see Definition~\ref{def.coal.Gal}). 

\begin{lemm}\label{lem.apx}~
\begin{zlist}
\item $\Delta(X)\subseteq C\ot X$, and consequently $X$ is a left $C$-comodule.
\item $A\cong P\,\Box_C X$.
\end{zlist}
\end{lemm}
\begin{proof}
For any $x\in X$,
$$
\begin{aligned}
(\id\ot\id\ot \pi)\circ(\id\ot \Delta)\circ \Delta(x) &= (\id\ot\id\ot \pi)\circ (\Delta \ot \id)\circ \Delta(x)\\
& = (\Delta \ot \id)\circ (\id\ot \pi)\circ  \Delta(x) = \Delta(x)\ot \bar e,
\end{aligned}
$$
by the coassociativity and the definition of $X$. Thus $\Delta(x) \in C\ot X$, which proves statement (1).

We will show next that the coaction $\varrho$ restricted to $A$ gives the isomorphism of assertion (2). By the coassociativity, $\varrho(A)\subseteq P\,\Box_C C$. For all $a\in A$, $\sum a\sw 0 \ot \pi(a\sw 1) = a\ot \bar{e}$, and thus applying $\varrho \ot \id$ to this equality we obtain 
$$
\begin{aligned}
\sum a\sw 0\ot a\sw 1 \ot \pi(a\sw 2) = \sum a\sw 0\ot a\sw 1 \ot \bar{e},
\end{aligned}
$$
 which, by the coassociativity again, yields $\varrho(A)\subseteq P\,\Box_C X$. Being a restriction of coaction $\varrho\!\mid_A$ is injective. Now, take any $\sum_i p^i\ot x^i\in P\,\Box_C X$ and set $a= \sum_i p^i\eps(x^i)$. Then
 $$
 \begin{aligned}
 \bar\varrho (a) = \sum_i p^i\sw 0\ot \pi(p^i\sw 1) \eps(x^i) = \sum_i p^i\ot \pi(x^i\sw 1) \eps(x^i\sw 2) = \sum_i p^i\ot x^i\ot \bar{e},
 \end{aligned}
 $$
 by the definitions of the cotensor product and $X$. Hence $a\in A$. Furthermore, again using the definition of the cotensor product we obtain
  $$
 \begin{aligned}
\varrho (a) = \sum_i p^i\sw 0\ot p^i\sw 1 \eps(x^i) = \sum_i p^i\ot x^i\sw 1 \eps(x^i\sw 2) = \sum_i p^i\ot x^i,
 \end{aligned}
 $$
 so $\varrho: A\to P\,\Box_C X$ is onto as well.
\end{proof}

\begin{lemm}\label{lem.proj}
Let $(P,C,\psi)$ be an entwining structure and assume that $P$ is an entwined module by multiplication and that $\varrho(1) = 1\ot e$.  
If $P$ is a $C$-equivariantly projective left $B$-module (see Definition~\ref{def.princ}), then $A$ is a projective left $B$-module with action given by the restriction of the multiplication.
\end{lemm}
\begin{proof}
Since, by the definition of $B$, the coaction $\varrho$ is left $B$-linear and $\bar{\varrho}$ is a compostion of $B$-module maps, for all $b\in B$ and $a\in A$, $\bar{\varrho}(ba) = ba\ot \bar{e}$, i.e.\ $ba\in A$. Let $\sigma: P\lra B\ot P$ be a right $C$-colinear left $B$-linear splitting of the multiplication map. Then a left $B$-module splitting of the multiplication map $B\ot A\to A$ is obtained as the composite
$$
\bar{\sigma}: \xymatrix{A\ar[r]^-{\cong}_-\varrho& P\,\Box_C X \ar[r]^-{\sigma\ot \id} & (B\ot P)\,\Box_C X \ar[r]^-{\cong}& B\ot (P\,\Box_C X)\ar[rr]^-{\cong}_-{\id\ot\id\ot\eps} && B\ot A},
$$
where $X$ is defined in Lemma~\ref{lem.apx}, and the isomorphisms are from there too. One easily checks that $\bar{\sigma} = \sigma\mid_A$, and thus the splitting property is immediate.
\end{proof}

\begin{lemm}\label{lem.pbax}
If $B\subseteq P$ is a principal $C$-Galois extension, then 
the restriction of the  Galois map $\can: P\ot_B P\lra P\ot C$ is an isomorphism
$$
\chi: \xymatrix{P\ot_B A\ar[r]^-{\cong} & P\ot X}
$$
of left $P$-modules.
\end{lemm}
\begin{proof}
We can consider the following chain of isomorphisms:
$$
\chi: \xymatrix{ P\ot_B A \ar[r]^-{\id\ot \varrho} & P\ot_B (P\,\Box_C X) \ar[r]^-\cong & (P\ot_B P)\,\Box_C X \ar[r]^-{\can\ot \id} &  (P\ot C) \,\Box_C X \ar[r]^-\cong & P\ot X},
$$
where the first isomorphism is from Lemma~\ref{lem.apx} and the second follows by the fact that $P$ is a projective hence flat right $B$-module -- a consequence of the fact that $P$ is a principal $C$-extension (see \cite[Theorem~2.5]{BrzHaj:Che}). The last isomorphism is obtained by applying counit in the middle factor. One can easily check that the resulting isomorphism is the restriction of the Galois map, as claimed.
\end{proof}

\begin{lemm}\label{lem.abar}
Let $B\subseteq P$ be a principal $C$-extension. Then 
\begin{zlist}
\item $P\ot X$ is a left $D$-comodule with coaction $\Lambda$ given by \eqref{mixed.coact}.
\item $\bar{A}\ot_B A\cong {}^\coD(P\ot X)_{\bar e}$,
where 
$$
\bar{A} = {}^\coD P_{\bar{e}} = \{a\in P\; |\; (\pi\ot \id)\circ \lambda (a) = \pi(e)\ot a\}.
$$
\end{zlist}
\end{lemm}
\begin{proof}
In terms of the explicit notation for the inverse of an entwining map \eqref{inv.alpha} the action of $\Lambda$ on $p\ot x\in P\ot X$ comes out as
$$
\Lambda(p\ot x) = \sum_A \pi(x\sw 1_A)\ot p^A \ot x\sw 2.
$$
Since $\Delta(X)\subseteq C\ot X$, the codomain of $\Lambda$ is $D\ot P\ot X$ as required. The map $\Lambda$ is counital, since $\pi$ is a coalgebra morphism and $\psi^{-1}$  preserves the counit in the sense of the second of equations \eqref{entw.inv.co}.  To check the coassociativity, take any $p\ot x\in P\ot X$, and compute
$$
\begin{aligned}
(\Delta_D\ot \id_{P\ot X})\circ \Lambda (p\ot x) &= \sum_A \Delta_D\left(\pi(x\sw 1_A)\right)\ot p^A \ot x\sw 2\\
&= \sum_A \pi(x\sw 1_A\sw 1)\ot \pi(x\sw 1_A\sw 2)\ot p^A \ot x\sw 2\\
&= \sum_{A,B} \pi(x\sw 1_A)\ot \pi(x\sw 2_B)\ot p^{AB} \ot x\sw 3,
\end{aligned}
$$
by the coalgebra homomorphism property of $\pi$ and the  the first of equations \eqref{entw.inv.co}.
On the other hand
$$
\begin{aligned}
(\id\ot \Lambda)\circ \Lambda (p\ot x) &=  \sum_A \pi(x\sw 1_A)\ot \Lambda(p^A \ot x\sw 2)\\
& = \sum_{A,B} \pi(x\sw 1_A)\ot \pi(x\sw 2_B)\ot p^{AB} \ot x\sw 3,
\end{aligned}
$$
 so that the map $\Lambda$ is coassociative as required. This proves assertion (1).

Note that the left $C$-coaction $\lambda: P\to C\ot P$ given by \eqref{left.coact} is right $B$-linear. Indeed, the contents of $B$ is fully characterised by $\psi(e\ot b) = b\ot e$, and hence equivalently by $e\ot b = \psi^{-1}(b\ot e)$. In particular, for all $p\in P$ and $b\in B$, 
$$
\begin{aligned}
\lambda(pb) &= \psi^{-1}(pb\ot e) = \sum_A e_A\ot (pb)^A\\
& = \sum_{A,B} e_{BA}\ot p^Ab^B = \sum_A e_A\ot p^Ab = \lambda(p)b,
\end{aligned}
$$
where the first of equations \eqref{entw.inv.al} has been used. Therefore, $P\ot_B A$ can be made into a left $D$-module by the coaction
$$
\begin{aligned}
\bar\Lambda = (\pi\ot\id_{P\ot_B A})\circ(\lambda\ot \id_A): P\ot_B A& \lra D\ot P\ot_B A,\\
  p\ot a&\lto \sum_A \pi(e_A)\ot p^A\ot a.
\end{aligned}
$$
We note next that the isomorphism $\chi$ of Lemma~\ref{lem.pbax} is a left $D$-comodule map. Indeed, for all $p\ot a\in P\ot_B A$, on one hand
$$
\begin{aligned}
(\id_D\ot \chi)\circ \bar\Lambda (p\ot a) &= \sum_A \pi(e_A)\ot p^Aa\sw 0\ot a\sw 1 = \sum_{A,\alpha} \pi(e_A)\ot p^Aa_\alpha\ot e^\alpha.
\end{aligned}
$$
On the other hand, using the definition of an entwining map and the resulting properties of its inverse \eqref{entw.inv}
$$
\begin{aligned}
\Lambda\circ\chi(p\ot a) &=\Lambda\left(\sum_\alpha pa_\alpha\ot e^\alpha\right)\\
&=\sum_{A,\alpha} \pi(e^\alpha\sw 1_A)\ot (pa_\alpha)^A \ot e^\alpha\sw 2\\
&=\sum_{A,\alpha,\beta} \pi(e^{\beta}{}_A)\ot (pa_{ \alpha\beta})^A \ot e^\alpha
\\
&=\sum_{A,B,\alpha,\beta} \pi(e^{\beta}{}_{BA})\ot p^Aa_{ \alpha\beta}{}^B \ot e^\alpha = \sum_{A,\alpha} \pi(e_A)\ot p^Aa_\alpha\ot e^\alpha.
\end{aligned}
$$
In addition to the axioms of an entwining map and \eqref{entw.inv}, the fact that $e$ is a group-like element and the first of equations \eqref{inv} have been used. 

Since $\chi$ is an isomorphism of left $D$-comodules and taking the coinvariants is a functor from the category of comodules to the category of vector spaces we obtain the isomorphism 
$$
{}^\coD (P\ot_B A)_{\bar{e}} \cong {}^\coD (P\ot X)_{\bar{e}}.
$$
Note that, for all left $D$-comodules $V$, the $\bar e$-coinvariants can be interpreted as the cotensor product
$$
{}^\coD V_{\bar{e}} \cong \k\,\Box_D V,
$$
where $\k$ is the right $D$-comodule by the coaction $1\mapsto 1\ot \bar{e}$. Since $A$ is a projective and hence flat  left $B$-module by Lemma~\ref{lem.proj}, 
$$
{}^\coD (P\ot_B A)_{\bar{e}} \cong \k\,\Box_D (P\ot_B A) \cong (\k\,\Box_D P)\ot_B A \cong {}^\coD P_{\bar{e}} \ot_B A = \bar{A}\ot_B A.
$$
This proves the second statement of the lemma.
\end{proof}

Put together, Lemmas~\ref{lem.apx}--\ref{lem.abar} yield all the assertions of Theorem~\ref{thm.main}.

We now move on to proving Theorem~\ref{thm.main.DK}. 

\begin{lemm}\label{lem.r}
Under the assumptions of Theorem~\ref{thm.main.DK},
$$
\varrho = (\id\ot r)\circ \delta,
$$
for  a right $H$-linear coalgebra map $r: H\lra C$.
\end{lemm}
\begin{proof}
Consider the following right $H$-linear map:
$$
r: H\lra C, \qquad h\mapsto e\cdot h.
$$
Since $C$ is a right $H$-module coalgebra, i.e.\ equations \eqref{mod.coal} are satisfied and $e$ is a group-like element we can compute, for all $h\in H$,
$$
\eps(r(h)) = \eps(e\cdot h) = \eps(e)\eps(h) = \eps(h),
$$
and
$$
\begin{aligned}
\Delta_C(r(h)) &= \sum (e\cdot h)\sw 1\ot (e\cdot h)\sw 2 = \sum e\cdot h\sw 1\ot e\cdot h\sw 2 = (r\ot r)\circ \Delta_H(h).
\end{aligned}
$$
Therefore, $r$ is a coalgebra morphism as required. As explained in Remark~\ref{rem.DK}, the entwining is of the Doi-Koppinen type (see \eqref{DK.entw}). Hence,
$$
\varrho(p) = \psi(e\ot p) = (1\ot e)\cdot \delta(p) = (\id\ot r)\circ \delta(p),
$$ 
as stated.
\end{proof}

In view of Lemma~\ref{lem.r}, in the setup of Theorem~\ref{thm.main.DK}, the coaction $\varrho: P\lra P\ot C$ can be seen as secondary to the coaction $\delta: P\lra P\ot H$. In what follows we will use $\sum p\sw 0\ot p\sw 1$ to denote $\delta(p)$. 

\begin{proof}(Theorem~\ref{thm.main.DK}) Since $\pi:C\lra D$ is a right $H$-linear map, for all $h\in H$,
$$
\pi\circ r (h) = \pi(e\cdot h) = \bar{e}\cdot h,
$$
and hence, for all $p\in P$,  
$$
\bar\varrho(p) = (\id\ot \pi)\circ (\id \ot r)\cdot \delta (p) = \sum p\sw 0 \ot \bar{e}\cdot p\sw 1.
$$
Therefore, for all $a,a'\in A$,
$$
\begin{aligned}
\bar\varrho(aa') &= \sum (aa')\sw 0 \ot \bar{e}\cdot (aa')\sw 1 = \sum a\sw 0 a'\sw 0 \ot \bar{e}\cdot (a\sw 1a'\sw 1)\\
&= \sum a\sw 0 a'\sw 0 \ot (\bar{e}\cdot a\sw 1)\cdot a'\sw 1 = \sum a a'\sw 0 \ot \bar{e} \cdot a'\sw 1 = aa'\ot \bar{e},
\end{aligned}
$$
since $\delta$ is an algebra map (that is, $P$ is a right $H$-comodule algebra). This proves the first statement of Theorem~\ref{thm.main.DK}, i.e.\ that $A$ is a subalgebra of $P$ (obviously $B\subseteq A$). 

Since the hypotheses of Theorem~\ref{thm.main} are satisfied, we know that $\bar{A}\ot_B A\cong {}^\coD(P\ot X)_{\bar e}$. First we calculate the form of $\bar{A}$. By the definition of $\bar{A}$ and in view of the form of the inverse of the Doi-Koppinen entwining in Remark~\ref{rem.DK},  $a\in \bar{A}$ if and only if 
$$
\bar{e}\ot a = \sum \bar{e}\cdot S^{-1}(a\sw 1)\ot a\sw 0.
$$
Applying $\id_D\ot \delta$ and acting with the right-most term (in $H$) on the left-most term (in $D$) one finds that
$$
\sum \bar{e}\cdot a\sw 1 \ot a\sw 0 = \sum \bar{e}\cdot (S^{-1}(a\sw 2)a\sw 1) \ot a\sw 0 = \bar{e} \ot a.
$$
Therefore, $\varrho(a) = a\ot \bar{e}$, i.e.\ $a\in A$. Similarly, if $a\in A$, then $\sum a\sw 0 \ot \bar{e}\cdot a\sw 1 = a\ot \bar{e}$, hence 
$$
\sum a\sw 0 \ot S^{-1}(a\sw 1) \ot \bar{e}\cdot a\sw 2 = \sum a\sw 0 \ot S^{-1}(a\sw 1) \ot \bar{e},
$$
which implies that
$$
\sum a\sw 0 \ot \bar{e}\cdot (a\sw 2S^{-1}(a\sw 1)) = \sum a\sw 0 \ot  \bar{e}\cdot S^{-1}(a\sw 1),
$$
i.e.\
$\bar{e}\ot a = \lambda (a)$. Therefore $a\in \bar{A}$. This shows that $\bar{A} = A$. 

Finally, using the explicit form of $\psi^{-1}$ we can observe that 
$\sum_ip^i\ot x^i \in {}^\coD(P\ot X)_{\bar e}$  if and only if
$$
\begin{aligned}
\sum_i \bar{e} \ot p^i\ot x^i  &= \sum_{i,A}\pi({x^i\sw 1}_A)\ot p^{iA}\ot x^i\sw 2\\
&= \sum_i \pi(x^i\sw 1\cdot S^{-1}(p^i\sw 1))\ot p^i\sw 0\ot x^i\sw 2\\
 &= \sum_i \pi(x^i\sw 1)\cdot S^{-1}(p^i\sw 1)\ot p^i\sw 0\ot x^i\sw 2.
\end{aligned}
$$
Therefore, 
$$
\sum_i \bar{e}\cdot p^i\sw 1\ot p^i\sw 0\ot x^i = \sum_i \pi(x^i\sw 1)\cdot (S^{-1}(p^i\sw 2)p^i\sw 1)\ot p^i\sw 0\ot x^i\sw 2,
$$
which, in view of Lemma~\ref{lem.r} and the antipode axioms, implies that
\begin{equation}\label{pdx}
\sum_i \bar{\varrho}(p^i)\ot x^i = \sum_i p^i \ot \pi(x^i\sw 1)\ot x^i\sw 2,
\end{equation}
i.e.\ $\sum_ip^i\ot x^i \in P\,\Box_D X$. Conversely, if equation \eqref{pdx} holds, then
$$
\sum_i (\bar{e}\cdot p^i\sw 2)\cdot S^{-1}(p^i\sw 1)\ot p^i\sw 0 \ot x^i = \sum_i \pi(x^i\sw 1)\cdot S^{-1}(p^i\sw 1)\ot p^i\sw 0 \ot x^i\sw 2,
$$
i.e.
$$
\sum_i \bar{e}\ot p^i \ot x^i = \sum_i \pi(x^i\sw 1\cdot S^{-1}(p^i\sw 1))\ot p^i\sw 0 \ot x^i\sw 2 = \Lambda(\sum_ip^i\ot x^i).
$$
Therefore, $\sum_ip^i\ot x^i\in  {}^\coD(P\ot X)_{\bar e}$ and we conclude that 
$$
A\ot_B A\cong P\,\Box_D X.
$$
Since the isomorphism in Theorem~\ref{thm.main} is given by the restriction of the canonical Galois map, so is the one above, as required.
\end{proof}

\end{document}